\documentclass[10pt]{amsart}
\usepackage[cp1251]{inputenc}
\usepackage[english,russian]{babel}
\usepackage{amsmath}
\usepackage{amssymb}
\usepackage{amsfonts}

\newtheorem{theorem}{Теорема}
\newtheorem{definition}{Определение}

\newtheorem{proposition}{Предложение}

\thanks{\Large Published in: Siberian Electronic Mathematical Reports, 2018,  Vol. 15, pp. 1--10.\\
DOI 10.17377/semi.2018.15.001\\
http://semr.math.nsc.ru/v15/p1-10.pdf}

\title[On fractional powers of the Bessel operator on a semiaxis]
{On fractional powers of the Bessel operator on a semiaxis}
\author[E.L. Shishkina, S.M. Sitnik]{Elina Shishkina\;and Sergei M. Sitnik}

\address{Elina Shishkina\newline
Voronezh State University, Voronezh,Russia.}
\email{ilina\_dico@mail.ru}

\address{Sergei M. Sitnik\newline
Belgorod State National Research University (BSU), Belgorod, Russia.}
\email{Sitnik@bsu.edu.ru\newline
\newline
~
----------------------------------------------------------------------------------\newline
~}

\address{Элина Шишкина\newline
Воронежский государственный университет, Воронеж, Россия.}
\email{ilina\_dico@mail.ru}

\address{Сергей Ситник\newline
Белгородский государственный национальный исследовательский университет\newline (БелГУ), Белгород, Россия.}
\email{Sitnik@bsu.edu.ru}



\begin{document}

\maketitle {\small

\begin{quote}
\noindent{\sc Abstract. } In this paper we study fractional powers of the Bessel differential operator defined on a semiaxis.  Some important properties of such fractional powers of the Bessel differential operator are proved. They include connections with Legendre functions for kernel representations, fractional integral operators of Liouville and Saigo, Mellin transform and index laws. Possible applications are indicated to differential equations with fractional powers of the Bessel differential operator.

\medskip

 \end{quote}
}

\section{Introduction}

We consider real powers of a singular Bessel operator
\begin{equation}\label{Bess}
B_\nu= D^2+\frac{\nu}{x}D,\qquad \nu\geq 0
\end{equation}
on a semiaxis $(0,\infty)$.

Let $\alpha>0$, $f(x)\in C^{[2\alpha]+1}(0,\infty)$.
We define fractional powers of a singular Bessel operator  on a semiaxis $(0,\infty)$ by the formulae (cf. [1]--[4], [18])
\begin{equation}\label{Bess3}
(B_{\nu,-}^{-\alpha}\,f)(x){=}\frac{1}{\Gamma(2\alpha)}\int\limits_x^{+\infty}\left(\frac{y^2-x^2}{2y}
\right)^{2\alpha-1}\,_2F_1\left(\alpha+\frac{\nu-1}{2},\alpha;2\alpha;1-\frac{x^2}{y^2}\right)f(y)dy.
\end{equation}
For brevity we shall also use a term fractional Bessel operator for \eqref{Bess3}.

In the paper we derive a formulae for \eqref{Bess3} via the Legendre functions, reduce it to the fractional Riemann--Liouville operator for  $\nu=0$, establish its connections with Saigo integral operator, find conditions under which \eqref{Bess3} is the left inverse for \eqref{Bess}. We also will find the Mellin transform for  \eqref{Bess3} and prove semigroup property for a fractional Bessel operator on a semiaxis.

Using a formulae connecting the Gauss hypergeometric function and the Legendre function of the form
$$
_2F_1(a,b;2b;z)=2^{2b-1}\Gamma\left(b+\frac{1}{2}\right)\, z^{\frac{1}{2}-b}(1-z)^{\frac{1}{2}\left(b-a-\frac{1}{2}\right)}P_{a-b-\frac{1}{2}}^{\frac{1}{2}-b}
\left[\left(1-\frac{z}{2}\right)\sqrt{1-z}\right]
$$
(cf. formulae 15.4.8, page 561, [5]), we derive
$$
\,_2F_1\left(\alpha+\frac{\nu-1}{2},\alpha;2\alpha;1-\frac{x^2}{y^2}\right)=
$$
$$
=2^{2\alpha-1}\Gamma\left(\alpha+\frac{1}{2}\right)\, \left(\frac{y^2-x^2}{y^2}\right)^{\frac{1}{2}-\alpha}\left(\frac{y}{x}
\right)^{\frac{\nu}{2}}P_{\frac{\nu}{2}-1}^{\frac{1}{2}-\alpha}
\left[\frac{1}{2}\left(\frac{x}{y}+\frac{y}{x}\right)\right]
$$
and so reduce \eqref{Bess3} to
$$
(B_{\nu,-}^{-\alpha}f)(x)
=\frac{\Gamma\left(\alpha+\frac{1}{2}\right)}{\Gamma(2\alpha)}\int\limits_x^b(y^2-x^2)^{\alpha-\frac{1}{2}}
\,\left(\frac{y}{x}\right)^{\frac{\nu}{2}}P_{\frac{\nu}{2}-1}^{\frac{1}{2}-\alpha}\left[\frac{1}{2}
\left(\frac{x}{y}+\frac{y}{x}\right)\right]f(y)dy.
$$

Reducing of fractional Bessel operators via the Legendre functions is a useful simplification as it is a representation via two--parameter Legendre functions instead of three--parameter Gauss hypergeometric functions.

In this paper we study the fractional Bessel operator on a semiaxis of the form \eqref{Bess3}, which is a generalization of the Liouville fractional integral. Another version of the fractional Bessel operator on a finite interval with an integral over $(0,x)$, which is a generalization of the Riemann--Liouville fractional integral, and other modifications were studied in [2]--[4].

We define \eqref{Bess3} under function conditions close to optimal. But below for simplicity we consider infinitely differentiable functions which are finite on  a semiaxis, namely with a support ${\rm supp}\,{f(x)}=[a,b], 0<a<b<\infty$.

\section{Basic properties of fractional powers of the Bessel operator on  a semiaxis}

In this section we obtain basic properties of operator \eqref{Bess3}, firstly, we demonstrate the connection between fractional  Bessel integral  on  a semiaxis with Liouville fractional integral and Saigo fractional integral, secondly, we show that when $\alpha=1$ the operator \eqref{Bess3} inverses the Bessel operator \eqref{Bess} with some additional conditions, and finally, we find  fractional  Bessel integral \eqref{Bess3} of power function.

\textbf{Property 1.} When $\nu=0$ fractional  Bessel integral  on  a semiaxis  $B_{0,-}^{-\alpha}$ is Liouville fractional integral
defined by the formula (5.3) p. 85 from [6]. Namely we have
$$
(B_{0,-}^{-\alpha}f)(x)=\frac{1}{\Gamma(2\alpha)}\int\limits_x^{\infty}(y-x)^{2\alpha-1}f(y)dy=(I_{-}^{2\alpha}f)(x).
$$

\textbf{Proof.} Actually we have
  $$(B_{0,-}^{-\alpha}f)(x)=\frac{1}{\Gamma(2\alpha)}\int\limits_x^\infty\left(\frac{y^2-x^2}{2y}\right)^{2\alpha-1}
  \,_2F_1\left(\alpha-\frac{1}{2},\alpha;2\alpha;1-\frac{x^2}{y^2}\right)f(y)dy.
 $$
Using a formulae following from  the integral representation for the Gauss hypergeometric function
$$
\,_2F_1\left(\alpha-\frac{1}{2},\alpha;2\alpha;1-\frac{x^2}{y^2}\right)=\left[\frac{2y}{x+y}\right]^{2\alpha-1},
$$
we derive
$$
(B_{0,-}^{-\alpha}f)(x)=\frac{1}{\Gamma(2\alpha)}\int\limits_x^\infty(y-x)^{2\alpha-1}f(y)dy=(I_{-}^{2\alpha}f)(x).
$$

\textbf{Property 2.} The next equality is valid
\begin{equation}\label{Prop2}
(B_{\nu,-}^{-\alpha}f)(x)=\frac{1}{2^{2\alpha}}J_{x^2}^{2\alpha,\frac{\nu{-}1}{2}-\alpha,-\alpha}\left(x^{\frac{\nu{-}1}{2}}f(\sqrt{x})\right),
\end{equation}
where
\begin{equation}\label{Saigo1}
J_{x}\,^{\gamma,\beta,\eta}f(x)=\frac{1}{\Gamma(\gamma)}\int\limits_x^\infty(t-x)^{\gamma-1}t^{-\gamma-\beta}\,_2F_1\left(\gamma+\beta,-\eta;\gamma;1-\frac{x}{t}\right)f(t)dt,
\end{equation}
is the Saigo fractional integral  (see [7], [9]). In \eqref{Saigo1} $\gamma>0,\beta,\theta$ are reals.

\textbf{Proof.}
Use change of a variable $y^2=t$ in the fractional Bessel operator \eqref{Bess3} we derive:
$$
(B_{\nu,-}^{-\alpha}f)(x)
=\frac{1}{2^{2\alpha}\Gamma(2\alpha)}\int\limits_{x^2}^\infty(t{-}x^2)^{2\alpha-1}t^{-\alpha}\,_2F_1\left(\alpha{+}\frac{\nu{-}1}{2},\alpha;2\alpha;1{-}\frac{x^2}{t}\right)f(\sqrt{t})dt.
$$
From comparing with  \eqref{Saigo1} it follows
$$
\gamma=2\alpha,\qquad \beta=\frac{\nu{-}1}{2}-\alpha,\qquad -\gamma-\beta=-\alpha-\frac{\nu{-}1}{2},\qquad \eta=-\alpha,
$$
that gives \eqref{Prop2}.

\textbf{Property 3.} If $\lim\limits_{x\rightarrow +\infty}g(x)=0$, $\lim\limits_{x\rightarrow +\infty}g'(x)=0$ then
$$
(B_{\nu,-}^{-1}B_\nu g)(x)=g(x).
$$
\textbf{Proof.}
Let consider $B_{\nu,b-}^{-\alpha}$ при $\alpha=1$:
$$
(B_{\nu,-}^{-1}f)(x)=\int\limits_x^\infty\left(\frac{y^2-x^2}{2y}\right)\,_2F_1\left(\frac{\nu+1}{2},1;2;1-\frac{x^2}{y^2}\right)f(y)dy.
$$
Using the equality for the Gauss hypergeometric function
$$
\,_2F_1\left(\frac{\nu+1}{2},1;2;1-\frac{x^2}{y^2}\right)=\frac{2}{1-\nu}\,\frac{y^2}{x^2-y^2}\left[\left(\frac{x}{y}\right)^{1-\nu}-1\right],
$$
the operator $B_{\nu,-}^{-1}$ may be written in the form
$$
(B_{\nu,-}^{-1}f)(x)=\frac{1}{\nu-1}\int\limits_x^\infty\, y\,\left[\left(\frac{x}{y}\right)^{1-\nu}-1\right] f(y)dy.
$$
Let $f(x)=B_\nu g(x)=g''(x)+\frac{\nu}{x}g'(x)$, then
$$
(B_{\nu,-}^{-1}f)(x)=(B_{\nu,-}^{-1}B_\nu g)(x)=$$
$$=\frac{1}{\nu-1}\int\limits_x^\infty y\left[\left(\frac{x}{y}\right)^{1-\nu}-1\right]\left(g''(y)+\frac{\nu}{y}g'(y)\right)dy=
$$
$$
=\frac{1}{\nu-1}\left[\int\limits_x^\infty y\left[\left(\frac{x}{y}\right)^{1-\nu}-1\right]g''(y)dy+\nu\int\limits_x^\infty\left[\left(\frac{x}{y}\right)^{1-\nu}-1\right]g'(y)dy\right].
$$
Twice differentiating the first term we derive
$$
\int\limits_x^\infty y\left[\left(\frac{x}{y}\right)^{1-\nu}{-}1\right]g''(y)dy{=}$$
$$=y\left[\left(\frac{x}{y}\right)^{1-\nu}{-}1\right]g'(y)\biggr|_{y=x}^{y=\infty}
{-}\int\limits_x^\infty(\nu x^{1-\nu}y^{\nu-1}{-}1)g'(y)dy=
$$
$$
=y\left[\left(\frac{x}{y}\right)^{1-\nu}{-}1\right]g'(y)\biggr|_{y=x}^{y=\infty}
-(\nu x^{1-\nu}y^{\nu-1}-1)g(y)\biggr|_{y=x}^{y=\infty}+$$
$$+\nu(\nu-1)x^{1-\nu}\int\limits_x^\infty y^{\nu-2}g(y)dy.
$$
Integrating by part the second term we obtain
$$
\int\limits_x^\infty\left[\left(\frac{x}{y}\right)^{1-\nu}-1\right]g'(y)dy=$$
$$=\left[\left(\frac{x}{y}\right)^{1-\nu}-1\right]g(y)
\biggr|_{y=x}^{y=\infty}-\frac{\nu-1}{x^{\nu-1}}\int\limits_x^\infty y^{\nu-2}g(y)dy.
$$
So it is obvious that for $\lim\limits_{x\rightarrow +\infty}g(x)=0$, $\lim\limits_{x\rightarrow +\infty}g'(x)=0$ it follows
$$
(B_{\nu,-}^{-1}B_\nu g)(x)=g(x).
$$

\textbf{Property 4.} When $x>0$ and $m+2\alpha+\nu<1$ we have a formula
\begin{equation}\label{Prop4}
B_{\nu,-}^{-\alpha}\,x^m=x^{2\alpha+m}\,2^{-2\alpha}\,\Gamma\left[
                                                           \begin{array}{cc}
                                                              $$-\alpha-\frac{m}{2},$$ & $$-\frac{\nu-1}{2}-\alpha-\frac{m}{2}$$ \\
                                                             $$\frac{1-\nu-m}{2},$$ & $$-\frac{m}{2}$$  \\
                                                           \end{array}
                                                         \right].
\end{equation}
\textbf{Proof.} Let us find a fractional Bessel operator acting on power function  $f(x){=}x^m$ for $x>0$ and $m+2\alpha+\nu<1$. Calculate
$$
B_{\nu,-}^{-\alpha}\,x^m=\frac{1}{\Gamma(2\alpha)}\int\limits_x^{\infty}\left(\frac{y^2-x^2}{2y}\right)^{2\alpha-1}\,_2F_1\left(\alpha+\frac{\nu-1}{2},\alpha;2\alpha;1-\frac{x^2}{y^2}\right)y^mdy=
$$
$$
=\left\{\frac{x^2}{y^2}=t,y=xt^{-\frac{1}{2}},dy=
-\frac{1}{2}xt^{-\frac{3}{2}}dt,y=x,t=1,y=\infty,t=0\right\}=$$
$${=}\frac{1}{2}\frac{1}{\Gamma(2\alpha)}\int\limits_0^{1}\left(\frac{x^2t^{-1}-x^2}{2xt^{-\frac{1}{2}}}\right)^{2\alpha-1}\,_2F_1\left(\alpha+\frac{\nu-1}{2},\alpha;2\alpha;1-t\right)(xt^{-\frac{1}{2}})^mxt^{-\frac{3}{2}}dt{=}
$$
$$=\frac{x^{2\alpha+m}}{2^{2\alpha}\Gamma(2\alpha)}\int\limits_0^{1}t^{-\alpha-\frac{m}{2}-1}(1-t)^{2\alpha-1}\,_2F_1\left(\alpha+\frac{\nu-1}{2},\alpha;2\alpha;1-t\right)dt.
$$
Now use the formulae 2.21.1.11, page 265, [8] of the form
$$
\int\limits_0^z x^{\mu-1}(z-x)^{c-1}\,_2F_1\left(a,b;c;1-\frac{x}{z}\right)dx=
$$
\begin{equation}\label{IR}
=z^{c+\mu-1}\Gamma\left[
                                                           \begin{array}{cc}
                                                             $$c,$$ & $$\mu,$$ \,\,\,\,\,\, $$c-a-b+\mu$$ \\
                                                             $$c-a+\mu,$$ & $$c-b+\mu$$  \\
                                                           \end{array}
                                                         \right],
\end{equation}
$$
z>0,\, {\rm Re}\,c>0,\, {\rm Re}\,(c-a-b+\mu)>0,
$$
for $B_{\nu,-}^{-\alpha}\,x^m$ we obtain
$$
z=1,\mu=-\alpha-\frac{m}{2},a=\alpha+\frac{\nu-1}{2},b=\alpha,c=2\alpha.
$$
Then, noticing that  $m+2\alpha+\nu<1$ we have
$$
  c-a-b+\mu=-\frac{\nu-1}{2}-\alpha-\frac{m}{2}>0,
$$
and, consequently, using \eqref{IR} we obtain
$$
B_{\nu,-}^{-\alpha}\,x^m=\frac{x^{2\alpha+m}}{2^{2\alpha}\Gamma(2\alpha)}\Gamma\left[
                                                           \begin{array}{cc}
                                                             $$2\alpha,$$ & $$-\alpha-\frac{m}{2},$$ \,\,\,\,\,\, $$-\frac{\nu-1}{2}-\alpha-\frac{m}{2}$$ \\
                                                             $$\frac{1-\nu-m}{2},$$ & $$-\frac{m}{2}$$  \\
                                                           \end{array}
                                                         \right].
$$
After simplifying we get \eqref{Prop4}.

\section{Integral Mellin transform of fractional Bessel integral  on  a semiaxis  and semigroup property}

In this section we obtain the formula of Mellin transform of fractional Bessel integral  on  a semiaxis $B_{\nu,-}^{-\alpha}$.

We use a standard definition of the Mellin transform for a function  $f$
 $$
 Mf(s)=f^*(s)=\int\limits_0^\infty x^{s-1}f(x)dx.
 $$

\textbf{Theorem 1.} Let $\alpha>0$. The Mellin transform applied to a fractional Bessel operator on a semiaxis $B_{\nu,-}^{-\alpha}$ is
\begin{equation}\label{Mellin2}
  ((B_{\nu,-}^{-\alpha}f)(x))^*(s)=\frac{1}{2^{2\alpha}}\,\,\Gamma\left[\begin{array}{cc}
$$\frac{s}{2},$$ & $$\frac{s}{2}-\frac{\nu-1}{2}$$ \\
$$\alpha+\frac{s}{2}-\frac{\nu-1}{2},$$ & $$\alpha+\frac{s}{2}$$  \\
                                                           \end{array}
                                                         \right] f^*(2\alpha+s).
\end{equation}
 \textbf{Proof.}
We have
$$
 ((B_{\nu,-}^{-\alpha}f)(x))^*(s)=\int\limits_0^\infty x^{s-1}(B_{\nu,-}^{-\alpha}f)(x)dx=\frac{1}{\Gamma(2\alpha)}\int\limits_0^\infty x^{s-1}dx\times$$
 $$\times\int\limits_x^{+\infty}\left(\frac{y^2-x^2}{2y}\right)^{2\alpha-1}\,_2F_1\left(\alpha+\frac{\nu-1}{2},\alpha;2\alpha;1-\frac{x^2}{y^2}\right)f(y)dy=
 $$
 $$
 =\frac{1}{\Gamma(2\alpha)}\int\limits_0^\infty f(y)(2y)^{1-2\alpha}dy\times
 $$
 $$\times\int\limits_0^y (y^2-x^2)^{2\alpha-1}\,_2F_1\left(\alpha+\frac{\nu-1}{2},\alpha;2\alpha;1-\frac{x^2}{y^2}\right) x^{s-1}dx.
 $$
First let find an internal integral
 $$
 I=\int\limits_0^y (y^2-x^2)^{2\alpha-1}\,_2F_1\left(\alpha+\frac{\nu-1}{2},\alpha;2\alpha;1-\frac{x^2}{y^2}\right) x^{s-1}dx=$$
 $$=\frac{1}{2}\int\limits_0^{y^2} (y^2-x)^{2\alpha-1}\,_2F_1\left(\alpha+\frac{\nu-1}{2},\alpha;2\alpha;1-\frac{x}{y^2}\right) x^{\frac{s}{2}-1}dx.
 $$
Then using \eqref{IR} we get
$$
z=y^2>0,\, c=2\alpha>0,\, a=\alpha+\frac{\nu-1}{2},\, b=\alpha,\,\alpha=\frac{s}{2},$$
$$c-a-b+\alpha=\frac{s}{2}-\frac{\nu-1}{2}>0\Rightarrow s>\nu-1.
$$
It follows
$$
I=\frac{y^{4\alpha+s-2}}{2}\,\,\Gamma\left[\begin{array}{cc}
$$2\alpha,$$ & $$\frac{s}{2},$$ \,\,\,\,\,\, $$\frac{s}{2}-\frac{\nu-1}{2}$$ \\
$$\alpha+\frac{s}{2}-\frac{\nu-1}{2},$$ & $$\alpha+\frac{s}{2}$$  \\
                                                           \end{array}
                                                         \right]
$$
and
  $$
 ((B_{\nu,-}^{-\alpha}f)(x))^*(s)=$$
 $$=\frac{1}{2}\,\,\Gamma\left[\begin{array}{cc}
 $$\frac{s}{2},$$ & $$\frac{s}{2}-\frac{\nu-1}{2}$$ \\
$$\alpha+\frac{s}{2}-\frac{\nu-1}{2},$$ & $$\alpha+\frac{s}{2}$$  \\
                                                           \end{array}
                                                         \right]\int\limits_0^\infty f(y)(2y)^{1-2\alpha}y^{4\alpha+s-2}dy=
 $$
$$
=\frac{1}{2^{2\alpha}}\,\,\Gamma\left[\begin{array}{cc}
 $$\frac{s}{2},$$ & $$\frac{s}{2}-\frac{\nu-1}{2}$$ \\
$$\alpha+\frac{s}{2}-\frac{\nu-1}{2},$$ & $$\alpha+\frac{s}{2}$$  \\
                                                           \end{array}
                                                         \right]\int\limits_0^\infty f(y)y^{2\alpha+s-1}dy=
$$
$$
=\frac{1}{2^{2\alpha}}\,\,\Gamma\left[\begin{array}{cc}
 $$\frac{s}{2},$$ & $$\frac{s}{2}-\frac{\nu-1}{2}$$ \\
$$\alpha+\frac{s}{2}-\frac{\nu-1}{2},$$ & $$\alpha+\frac{s}{2}$$  \\
                                                           \end{array}
                                                         \right] f^*(2\alpha+s).
$$
That completes the proof of the Theorem 1.

\textbf{Theorem 2.}  For $\alpha,\beta>0$  the semigroup property for the fractional Bessel integral is true
\begin{equation}\label{SemiGroup3}
    B_{\nu,-}^{-\alpha}B_{\nu,-}^{-\beta}f= B_{\nu,-}^{-(\alpha+\beta)}f.
\end{equation}

 \textbf{Proof.} Let $g(y)=(B_{\nu,-}^{-\beta}f)(y)$. Using \eqref{Mellin2} we derive
$$
 (B_{\nu,-}^{-\alpha}[(B_{\nu,-}^{-\beta}f)(y)](x))^*(s)=((B_{\nu,-}^{-\alpha}g)(x))^*(s)=
 $$
 $$
 =\frac{1}{2^{2\alpha}}\,\,\Gamma\left[\begin{array}{cc}
  $$\frac{s}{2},$$ & $$\frac{s}{2}-\frac{\nu-1}{2}$$ \\
$$\alpha+\frac{s}{2}-\frac{\nu-1}{2},$$ & $$\alpha+\frac{s}{2}$$  \\
                                                           \end{array}
                                                         \right] g^*(2\alpha+s)=
 $$
 $$
 = \frac{1}{2^{2\alpha}}\,\,\Gamma\left[\begin{array}{cc}
 $$\frac{s}{2},$$ & $$\frac{s}{2}-\frac{\nu-1}{2}$$ \\
$$\alpha+\frac{s}{2}-\frac{\nu-1}{2},$$ & $$\alpha+\frac{s}{2}$$  \\
                                                           \end{array}
                                                         \right]
 (B_{\nu,-}^{-\beta}f)^*(2\alpha+s)=
$$
 $$
 = \frac{1}{2^{2\alpha}}\,\,\Gamma\left[\begin{array}{cc}
 $$\frac{s}{2},$$ & $$\frac{s}{2}-\frac{\nu-1}{2}$$ \\
$$\alpha+\frac{s}{2}-\frac{\nu-1}{2},$$ & $$\alpha+\frac{s}{2}$$  \\
                                                           \end{array}
                                                         \right]\times$$
                                                         $$\times\frac{1}{2^{2\beta}}\,\,\Gamma\left[\begin{array}{cc}
 $$\frac{2\alpha+s}{2},$$ & $$\frac{2\alpha+s}{2}-\frac{\nu-1}{2}$$ \\
$$\beta+\frac{2\alpha+s}{2}-\frac{\nu-1}{2},$$ & $$\beta+\frac{2\alpha+s}{2}$$  \\
                                                           \end{array}
                                                         \right]
f^*(2\alpha+2\beta+s).
$$
From the other side
$$
(B_{\nu,-}^{-(\alpha+\beta)}f)^*(s)=$$
$$=\frac{1}{2^{2(\alpha+\beta)}}\,\,\Gamma\left[\begin{array}{cc}
 $$\frac{s}{2},$$ & $$\frac{s}{2}-\frac{\nu-1}{2}$$ \\
$$\alpha+\beta+\frac{s}{2}-\frac{\nu-1}{2},$$ & $$\alpha+\beta+\frac{s}{2}$$  \\
                                                           \end{array}
                                                         \right] f^*(2\alpha+2\beta+s).
$$
And use the formulae
$$
\frac{1}{2^{2\alpha}}\Gamma\left[\begin{array}{cc}
 $$\frac{s}{2},$$ & $$\frac{s}{2}-\frac{\nu-1}{2}$$ \\
$$\alpha+\frac{s}{2}-\frac{\nu-1}{2},$$ & $$\alpha+\frac{s}{2}$$  \\
                                                           \end{array}
                                                         \right]{\times}\frac{1}{2^{2\beta}}\,\,\Gamma\left[\begin{array}{cc}
 $$\frac{2\alpha+s}{2},$$ & $$\frac{2\alpha+s}{2}-\frac{\nu-1}{2}$$ \\
$$\beta+\frac{2\alpha+s}{2}-\frac{\nu-1}{2},$$ & $$\beta+\frac{2\alpha+s}{2}$$  \\
                                                           \end{array}
                                                         \right]{=}
$$
$$
{=}\frac{1}{2^{2(\alpha+\beta)}}\,\,\Gamma\left[\begin{array}{cc}
 $$\frac{s}{2},$$ & $$\frac{s}{2}-\frac{\nu-1}{2}$$ \\
$$\alpha+\beta+\frac{s}{2}-\frac{\nu-1}{2},$$ & $$\alpha+\beta+\frac{s}{2}$$  \\
                                                           \end{array}
                                                         \right].
$$
we prove the semigroup property \eqref{SemiGroup3}, which sometimes is called "the index law".

As a conclusion let us mention that fractional Bessel operators of the form \eqref{Bess3} may be applied to a study of fractional order differential equations of different types. They are also important in the theory of transmutations, cf. [10]--[15].

Differential equations with formal powers of Bessel operators are studied as models for particle random walks [16]--[17]. Different powers of hyperbolic operators connected to Bessel operators are studied in [19]--[23], cf. also [24]--[25].

\begin{center}

~

REFERENCES

~

\end{center}

\newpage

\begin{center}
\LARGE
О ДРОБНЫХ СТЕПЕНЯХ ОПЕРАТОРА БЕССЕЛЯ НА ПОЛУОСИ\\
Ситник С.М., Шишкина Э.Л.
\end{center}

~

\thanks{Опубликовано: Сибирские Электронные Математические Известия, 2018, Том 15 (2018), С. 1--10.\\
DOI 10.17377/semi.2018.15.001\\
http://semr.math.nsc.ru/v15/p1-10.pdf}
\section{Введение}

~

\begin{quote}
\noindent{\sc Аннотация. } В статье рассматриваются явные интегральные представления для дробных степеней оператора Бесселя на полуоси. Для них доказаны представления ядер через функции Лежандра, установлены связи с операторами дробного интегродифференцирования Римана--Лиувилля и Сайго, вычислено преобразование Меллина и установлено полугрупповое свойство (индексные законы). Кратко намечены приложения к дифференциальным уравнениям дробного порядка с дробными степенями операторов Бесселя.
\medskip
\end{quote}

 Мы рассматриваем вещественные степени сингулярного дифференциального оператора Бесселя
\begin{equation}\label{Bess}
B_\nu= D^2+\frac{\nu}{x}D,\qquad \nu\geq 0
\end{equation}
на вещественной полуоси $(0,\infty)$.

\begin{definition}
Пусть $\alpha>0$, $f(x)\in C^{[2\alpha]+1}(0,\infty)$. Дробную степень оператора Бесселя на полуоси $(0,\infty)$ определим, следуя работам  \cite{Ida}--\cite{Kosovo}, формулой
$$
(IB_{\nu,-}^{\alpha}\,f)(x){=}
$$
\begin{equation}\label{Bess3}
=\frac{1}{\Gamma(2\alpha)}\int\limits_x^{+\infty}\left(\frac{y^2-x^2}{2y}
\right)^{2\alpha-1}\,_2F_1\left(\alpha+\frac{\nu-1}{2},\alpha;2\alpha;1-\frac{x^2}{y^2}\right)f(y)dy.
\end{equation}
Для краткости будем также называть выражение \eqref{Bess3} \textbf{дробным интегралом Бесселя на полуоси}.
\end{definition}

В работе \cite{McBride}  введены пространства, приспособленные для работы с операторами вида \eqref{Bess3}:
$$
F_p=\left\{\varphi\in C^\infty(0,\infty):x^k\frac{d^k\varphi}{dx^k}\in L^p(0,\infty)\,\, {\text{для}}\,k=0,1,2,...\right\},\qquad 1\leq p<\infty,
$$
$$
F_\infty=\left\{\varphi\in C^\infty(0,\infty):x^k\frac{d^k\varphi}{dx^k}\rightarrow0 \,\, {\text{при}}\, x\rightarrow0+
 {\text{,\,\,и при}}\, x\rightarrow\infty\,{\text{для}}\,k=0,1,2,...\right\}
$$
и
$$
F_{p,\mu}=\left\{\varphi: x^{-\mu}\varphi(x)\in F_p\right\},\qquad 1\leq p\leq \infty,\qquad \mu\in\mathbb{C}.
$$
Кроме того, в \cite{McBride}  доказано, что \eqref{Bess3} имеет обратный оператор.

\begin{definition}
 \textbf{Дробную производную Бесселя на полуоси} определим равенством
\begin{equation}\label{DrobessDer}
(DB_{\nu,-}^\alpha f)(x)=B_\nu^n(IB_{\nu,-}^{n-\alpha}f)(x),\qquad \alpha>0.
\end{equation}
\end{definition}

Можно показать, что на подходящем классе функций $DB_{\nu,-}^\alpha$ есть левый обратный оператор к $IB_{\nu,-}^\alpha$.

В этой работе мы получим для интегрального оператора \eqref{Bess3} формулу, выражающую его ядро через функции Лежандра; сведем   дробный интеграл Бесселя на полуоси к дробному интегралу Лиувилля при $\nu=0$, а также установим связь дробного интеграла Бесселя \eqref{Bess3} с дробным интегралом Сайго; установим условия, при которых оператор \eqref{Bess3} есть левый обратный к \eqref{Bess}; найдем преобразование Меллина оператора \eqref{Bess3} и докажем полугрупповое свойство для дробного интеграла Бесселя на полуоси; а также решим уравнение с дробной производной Бесселя на полуоси.

 Используя формулу, связывающую гипергеометрическую функцию Гаусса и функцию Лежандра вида
$$
_2F_1(a,b;2b;z)=2^{2b-1}\Gamma\left(b+\frac{1}{2}\right)\, z^{\frac{1}{2}-b}(1-z)^{\frac{1}{2}\left(b-a-\frac{1}{2}\right)}P_{a-b-\frac{1}{2}}^{\frac{1}{2}-b}
\left[\left(1-\frac{z}{2}\right)\sqrt{1-z}\right]
$$
(см. формулу 15.4.8 на стр. 561 из \cite{Abramowitz}), мы получим
$$
\,_2F_1\left(\alpha+\frac{\nu-1}{2},\alpha;2\alpha;1-\frac{x^2}{y^2}\right)=
$$
$$
=2^{2\alpha-1}\Gamma\left(\alpha+\frac{1}{2}\right)\, \left(\frac{y^2-x^2}{y^2}\right)^{\frac{1}{2}-\alpha}\left(\frac{y}{x}
\right)^{\frac{\nu}{2}}P_{\frac{\nu}{2}-1}^{\frac{1}{2}-\alpha}
\left[\frac{1}{2}\left(\frac{x}{y}+\frac{y}{x}\right)\right]
$$
и сможем записать \eqref{Bess3} в виде
$$
(B_{\nu,-}^{-\alpha}f)(x)
=\frac{\Gamma\left(\alpha+\frac{1}{2}\right)}{\Gamma(2\alpha)}\int\limits_x^b(y^2-x^2)^{\alpha-\frac{1}{2}}
\,\left(\frac{y}{x}\right)^{\frac{\nu}{2}}P_{\frac{\nu}{2}-1}^{\frac{1}{2}-\alpha}\left[\frac{1}{2}
\left(\frac{x}{y}+\frac{y}{x}\right)\right]f(y)dy.
$$

Выражение дробных интегралов Бесселя через функции Лежандра является полезным и является упрощением первоначального определения, так как гипергеометрическая функция Гаусса зависит от трёх параметров, а функция Лежандра --- от двух.

В статье изучается дробный интеграл Бесселя в форме \eqref{Bess3}, который в частном случае соответствует дробному интегралу Лиувилля. Существует также версия дробного интеграла Бесселя на конечном отрезке с интегрированием по промежутку $(0,x)$, которая в частном случае является дробным интегралом Римана--Лиувилля, а также их дальнейшие модификации, см. \cite{Sita1}--\cite{Kosovo}.

Определение \eqref{Bess3} дано выше при ограничениях на функцию, близким к оптимальным в указанном классе, однако далее мы будем для простоты предполагать, что рассматриваются бесконечно дифференцируемые функции, финитные на полуоси, то есть их носитель  есть ${\rm supp}\,{f(x)}=[a,b], 0<a<b<\infty$.

\section{Основные свойства дробных степей оператора Бесселя на полуоси}

В этом пункте мы получим основные свойства оператора \eqref{Bess3}, во-первых, демонстрирующие связь дробного интеграла Бесселя на полуоси с дробным интегралом Лиувилля и с дробным интегралом Сайго, во-вторых, показывающие, что при дополнительных условиях оператор \eqref{Bess3} при $\alpha=1$ обращает оператор Бесселя \eqref{Bess}, и, наконец, найдем дробный интеграл Бесселя на полуоси \eqref{Bess3} от степенной функции.

\begin{proposition} При $\nu=0$ дробный интеграл Бесселя на полуоси $B_{0,-}^{-\alpha}$ сводится к  дробному интегралу Лиувилля, определённому формулой (5.3) стр. 85 из \cite{KK}, а именно, справедлива формула
$$
(IB_{0,-}^{\alpha}f)(x)=\frac{1}{\Gamma(2\alpha)}\int\limits_x^{\infty}(y-x)^{2\alpha-1}f(y)dy=(I_{-}^{2\alpha}f)(x).
$$
\end{proposition}
\begin{proof} Действительно, имеем
  $$(IB_{0,-}^{\alpha}f)(x)=\frac{1}{\Gamma(2\alpha)}\int\limits_x^\infty\left(\frac{y^2-x^2}{2y}\right)^{2\alpha-1}
  \,_2F_1\left(\alpha-\frac{1}{2},\alpha;2\alpha;1-\frac{x^2}{y^2}\right)f(y)dy.
 $$
Используя формулу, которая  получается из интегрального представления гипергеометрической функции Гаусса
$$
\,_2F_1\left(\alpha-\frac{1}{2},\alpha;2\alpha;1-\frac{x^2}{y^2}\right)=\left[\frac{2y}{x+y}\right]^{2\alpha-1},
$$
мы получим
$$
(IB_{0,-}^{\alpha}f)(x)=\frac{1}{\Gamma(2\alpha)}\int\limits_x^\infty(y-x)^{2\alpha-1}f(y)dy=(I_{-}^{2\alpha}f)(x).
$$
\end{proof}

\begin{proposition} Имеет место равенство
\begin{equation}\label{Prop2}
(IB_{\nu,-}^{\alpha}f)(x)=\frac{1}{2^{2\alpha}}J_{x^2}^{2\alpha,\frac{\nu{-}1}{2}-\alpha,-\alpha}\left(x^{\frac{\nu{-}1}{2}}f(\sqrt{x})\right),
\end{equation}
где
\begin{equation}\label{Saigo1}
J_{x}\,^{\gamma,\beta,\eta}f(x)=\frac{1}{\Gamma(\gamma)}\int\limits_x^\infty(t-x)^{\gamma-1}t^{-\gamma-\beta}\,_2F_1\left(\gamma+\beta,-\eta;\gamma;1-\frac{x}{t}\right)f(t)dt,
\end{equation}
--- дробный интеграл Сайго   (см. \cite{Saigo}, \cite{Repin}). В \eqref{Saigo1} $\gamma>0,\beta,\theta$ --- вещественные числа.
\end{proposition}
\begin{proof}
Произведем замену переменной $y^2=t$ в дробном интеграле Бесселя на полуоси \eqref{Bess3}, получим:
$$
(IB_{\nu,-}^{\alpha}f)(x)
=\frac{1}{2^{2\alpha}\Gamma(2\alpha)}\int\limits_{x^2}^\infty(t{-}x^2)^{2\alpha-1}t^{-\alpha}\,_2F_1\left(\alpha{+}\frac{\nu{-}1}{2},\alpha;2\alpha;1{-}\frac{x^2}{t}\right)f(\sqrt{t})dt.
$$
Сравнивая полученное выражение с \eqref{Saigo1}, будем иметь
$$
\gamma=2\alpha,\qquad \beta=\frac{\nu{-}1}{2}-\alpha,\qquad -\gamma-\beta=-\alpha-\frac{\nu{-}1}{2},\qquad \eta=-\alpha,
$$
что и дает \eqref{Prop2}.
\end{proof}

\begin{proposition} При $\lim\limits_{x\rightarrow +\infty}g(x)=0$, $\lim\limits_{x\rightarrow +\infty}g'(x)=0$ получим, что
$$
(IB_{\nu,-}^{-1}B_\nu g)(x)=g(x).
$$
\end{proposition}
\begin{proof}
Рассмотрим $IB_{\nu,-}^{\alpha}$ при $\alpha=1$:
$$
(IB_{\nu,-}^{1}f)(x)=\int\limits_x^\infty\left(\frac{y^2-x^2}{2y}\right)\,_2F_1\left(\frac{\nu+1}{2},1;2;1-\frac{x^2}{y^2}\right)f(y)dy.
$$
Поскольку для гипергеометрической функции Гаусса справедливо равенство
$$
\,_2F_1\left(\frac{\nu+1}{2},1;2;1-\frac{x^2}{y^2}\right)=\frac{2}{1-\nu}\,\frac{y^2}{x^2-y^2}\left[\left(\frac{x}{y}\right)^{1-\nu}-1\right],
$$
то $IB_{\nu,-}^{1}$ можно записать в виде
$$
(IB_{\nu,-}^{1}f)(x)=\frac{1}{\nu-1}\int\limits_x^\infty\, y\,\left[\left(\frac{x}{y}\right)^{1-\nu}-1\right] f(y)dy.
$$
Пусть $f(x)=B_\nu g(x)=g''(x)+\frac{\nu}{x}g'(x)$, тогда
$$
(IB_{\nu,-}^{1}f)(x)=(B_{\nu,-}^{-1}B_\nu g)(x)=$$
$$=\frac{1}{\nu-1}\int\limits_x^\infty y\left[\left(\frac{x}{y}\right)^{1-\nu}-1\right]\left(g''(y)+\frac{\nu}{y}g'(y)\right)dy=
$$
$$
=\frac{1}{\nu-1}\left[\int\limits_x^\infty y\left[\left(\frac{x}{y}\right)^{1-\nu}-1\right]g''(y)dy+\nu\int\limits_x^\infty\left[\left(\frac{x}{y}\right)^{1-\nu}-1\right]g'(y)dy\right].
$$
Дважды интегрируя по частям первое слагаемое, получим
$$
\int\limits_x^\infty y\left[\left(\frac{x}{y}\right)^{1-\nu}{-}1\right]g''(y)dy{=}$$
$$=y\left[\left(\frac{x}{y}\right)^{1-\nu}{-}1\right]g'(y)\biggr|_{y=x}^{y=\infty}
{-}\int\limits_x^\infty(\nu x^{1-\nu}y^{\nu-1}{-}1)g'(y)dy=
$$
$$
=y\left[\left(\frac{x}{y}\right)^{1-\nu}{-}1\right]g'(y)\biggr|_{y=x}^{y=\infty}
-(\nu x^{1-\nu}y^{\nu-1}-1)g(y)\biggr|_{y=x}^{y=\infty}+$$
$$+\nu(\nu-1)x^{1-\nu}\int\limits_x^\infty y^{\nu-2}g(y)dy.
$$
Интегрируя по частям второе слагаемое, будем иметь
$$
\int\limits_x^\infty\left[\left(\frac{x}{y}\right)^{1-\nu}-1\right]g'(y)dy=$$
$$=\left[\left(\frac{x}{y}\right)^{1-\nu}-1\right]g(y)
\biggr|_{y=x}^{y=\infty}-\frac{\nu-1}{x^{\nu-1}}\int\limits_x^\infty y^{\nu-2}g(y)dy.
$$
Тогда, очевидно, что при $\lim\limits_{x\rightarrow +\infty}g(x)=0$, $\lim\limits_{x\rightarrow +\infty}g'(x)=0$ получим
$$
(IB_{\nu,-}^{1}B_\nu g)(x)=g(x).
$$
\end{proof}

\begin{proposition} При $x>0$ и $m+2\alpha+\nu<1$ справедлива формула
\begin{equation}\label{Prop4}
IB_{\nu,-}^{\alpha}\,x^m=2^{-2\alpha}\,\Gamma\left[
                                                           \begin{array}{cc}
                                                              $$-\alpha-\frac{m}{2},$$ & $$-\frac{\nu-1}{2}-\alpha-\frac{m}{2}$$ \\
                                                             $$\frac{1-\nu-m}{2},$$ & $$-\frac{m}{2}$$  \\
                                                           \end{array}
                                                         \right]\, x^{2\alpha+m}.
\end{equation}
\end{proposition}
\begin{proof} Найдем дробный интеграл Бесселя на полуоси от степенной функции  $f(x){=}x^m$ при $x>0$ и $m+2\alpha+\nu<1$. Имеем
$$
IB_{\nu,-}^{\alpha}\,x^m=\frac{1}{\Gamma(2\alpha)}\int\limits_x^{\infty}\left(\frac{y^2-x^2}{2y}\right)^{2\alpha-1}\,_2F_1\left(\alpha+\frac{\nu-1}{2},\alpha;2\alpha;1-\frac{x^2}{y^2}\right)y^mdy=
$$
$$
=\left\{\frac{x^2}{y^2}=t,y=xt^{-\frac{1}{2}},dy=
-\frac{1}{2}xt^{-\frac{3}{2}}dt,y=x,t=1,y=\infty,t=0\right\}=$$
$${=}\frac{1}{2}\frac{1}{\Gamma(2\alpha)}\int\limits_0^{1}\left(\frac{x^2t^{-1}-x^2}{2xt^{-\frac{1}{2}}}\right)^{2\alpha-1}\,_2F_1\left(\alpha+\frac{\nu-1}{2},\alpha;2\alpha;1-t\right)(xt^{-\frac{1}{2}})^mxt^{-\frac{3}{2}}dt{=}
$$
$$=\frac{x^{2\alpha+m}}{2^{2\alpha}\Gamma(2\alpha)}\int\limits_0^{1}t^{-\alpha-\frac{m}{2}-1}(1-t)^{2\alpha-1}\,_2F_1\left(\alpha+\frac{\nu-1}{2},\alpha;2\alpha;1-t\right)dt.
$$
Используя формулу 2.21.1.11 из \cite{IR3}, стр. 265 вида
$$
\int\limits_0^z x^{\mu-1}(z-x)^{c-1}\,_2F_1\left(a,b;c;1-\frac{x}{z}\right)dx=
$$
\begin{equation}\label{IR}
=z^{c+\mu-1}\Gamma\left[
                                                           \begin{array}{cc}
                                                             $$c,$$ & $$\mu,$$ \,\,\,\,\,\, $$c-a-b+\mu$$ \\
                                                             $$c-a+\mu,$$ & $$c-b+\mu$$  \\
                                                           \end{array}
                                                         \right],
\end{equation}
$$
z>0,\, {\rm Re}\,c>0,\, {\rm Re}\,(c-a-b+\mu)>0,
$$
для $B_{\nu,-}^{-\alpha}\,x^m$ мы будем иметь
$$
z=1,\mu=-\alpha-\frac{m}{2},a=\alpha+\frac{\nu-1}{2},b=\alpha,c=2\alpha.
$$
Тогда, учитывая, что $m+2\alpha+\nu<1$  справедливо неравенство
$$
  c-a-b+\mu=-\frac{\nu-1}{2}-\alpha-\frac{m}{2}>0,
$$
и, следовательно, воспользовавшись \eqref{IR}, получим
$$
IB_{\nu,-}^{\alpha}\,x^m=\frac{x^{2\alpha+m}}{2^{2\alpha}\Gamma(2\alpha)}\Gamma\left[
                                                           \begin{array}{cc}
                                                             $$2\alpha,$$ & $$-\alpha-\frac{m}{2},$$ \,\,\,\,\,\, $$-\frac{\nu-1}{2}-\alpha-\frac{m}{2}$$ \\
                                                             $$\frac{1-\nu-m}{2},$$ & $$-\frac{m}{2}$$  \\
                                                           \end{array}
                                                         \right]
$$
что после упрощения дает \eqref{Prop4}.
\end{proof}

Отметим важность полученной формулы, устанавливающей, что дробный интеграл Бесселя переводит одну степенную функцию в другую, так как это позволяет распространить его определение на произвольные степенные ряды. Явный вид константы в формуле \eqref{Prop4} в форме отношения гамма--функций показывает, что дробный интеграл Бесселя является оператором дробного дифференцирования типа Гельфонда--Леонтьева \cite{KK}.

\section{Интегральное преобразование Меллина дробных степей оператора Бесселя на полуоси и полугруповое свойство}

В этом пункте выведем  формулу преобразования Меллина от дробного интеграла Бесселя на полуоси $B_{\nu,-}^{-\alpha}$.

Преобразование Меллина функции $f$ определяется формулой
 $$
 Mf(s)=f^*(s)=\int\limits_0^\infty x^{s-1}f(x)dx.
 $$

\begin{theorem} Пусть $\alpha>0$. Преобразования Меллина от дробного интеграла и дробной производной Бесселя на полуоси имеют вид
\begin{equation}\label{Mellin2}
  ((IB_{\nu,-}^{\alpha}f)(x))^*(s)=\frac{1}{2^{2\alpha}}\,\,\Gamma\left[\begin{array}{cc}
$$\frac{s}{2},$$ & $$\frac{s}{2}-\frac{\nu-1}{2}$$ \\
$$\alpha+\frac{s}{2}-\frac{\nu-1}{2},$$ & $$\alpha+\frac{s}{2}$$  \\
                                                           \end{array}
                                                         \right] f^*(2\alpha+s),
\end{equation}
\begin{equation}\label{MellinD}
  ((DB_{\nu,-}^{\alpha}f)(x))^*(s)=2^{2\alpha}\,\Gamma\left[\begin{array}{cc}
$$\frac{s}{2},$$ & $$\frac{s}{2}-\frac{\nu-1}{2}$$ \\
$$\frac{s}{2}-\alpha-\frac{\nu-1}{2},$$ & $$\frac{s}{2}-\alpha$$  \\
\end{array} \right] f^*(s-2\alpha).
\end{equation}
\end{theorem}
\begin{proof}
Имеем
$$
 ((IB_{\nu,-}^{\alpha}f)(x))^*(s)=\int\limits_0^\infty x^{s-1}(IB_{\nu,-}^{\alpha}f)(x)dx=\frac{1}{\Gamma(2\alpha)}\int\limits_0^\infty x^{s-1}dx\times$$
 $$\times\int\limits_x^{+\infty}\left(\frac{y^2-x^2}{2y}\right)^{2\alpha-1}\,_2F_1\left(\alpha+\frac{\nu-1}{2},\alpha;2\alpha;1-\frac{x^2}{y^2}\right)f(y)dy=
 $$
 $$
 =\frac{1}{\Gamma(2\alpha)}\int\limits_0^\infty f(y)(2y)^{1-2\alpha}dy\times
 $$
 $$\times\int\limits_0^y (y^2-x^2)^{2\alpha-1}\,_2F_1\left(\alpha+\frac{\nu-1}{2},\alpha;2\alpha;1-\frac{x^2}{y^2}\right) x^{s-1}dx.
 $$
Найдем внутренний интеграл
 $$
 I=\int\limits_0^y (y^2-x^2)^{2\alpha-1}\,_2F_1\left(\alpha+\frac{\nu-1}{2},\alpha;2\alpha;1-\frac{x^2}{y^2}\right) x^{s-1}dx=$$
 $$=\frac{1}{2}\int\limits_0^{y^2} (y^2-x)^{2\alpha-1}\,_2F_1\left(\alpha+\frac{\nu-1}{2},\alpha;2\alpha;1-\frac{x}{y^2}\right) x^{\frac{s}{2}-1}dx.
 $$
Используя  формулу \eqref{IR}, получим
$$
z=y^2>0,\, c=2\alpha>0,\, a=\alpha+\frac{\nu-1}{2},\, b=\alpha,\,\alpha=\frac{s}{2},$$
$$c-a-b+\alpha=\frac{s}{2}-\frac{\nu-1}{2}>0\Rightarrow s>\nu-1.
$$
Тогда
$$
I=\frac{y^{4\alpha+s-2}}{2}\,\,\Gamma\left[\begin{array}{cc}
$$2\alpha,$$ & $$\frac{s}{2},$$ \,\,\,\,\,\, $$\frac{s}{2}-\frac{\nu-1}{2}$$ \\
$$\alpha+\frac{s}{2}-\frac{\nu-1}{2},$$ & $$\alpha+\frac{s}{2}$$  \\
                                                           \end{array}
                                                         \right]
$$
и
  $$
 ((IB_{\nu,-}^{\alpha}f)(x))^*(s)=$$
 $$=\frac{1}{2}\,\,\Gamma\left[\begin{array}{cc}
 $$\frac{s}{2},$$ & $$\frac{s}{2}-\frac{\nu-1}{2}$$ \\
$$\alpha+\frac{s}{2}-\frac{\nu-1}{2},$$ & $$\alpha+\frac{s}{2}$$  \\
                                                           \end{array}
                                                         \right]\int\limits_0^\infty f(y)(2y)^{1-2\alpha}y^{4\alpha+s-2}dy=
 $$
$$
=\frac{1}{2^{2\alpha}}\,\,\Gamma\left[\begin{array}{cc}
 $$\frac{s}{2},$$ & $$\frac{s}{2}-\frac{\nu-1}{2}$$ \\
$$\alpha+\frac{s}{2}-\frac{\nu-1}{2},$$ & $$\alpha+\frac{s}{2}$$  \\
                                                           \end{array}
                                                         \right]\int\limits_0^\infty f(y)y^{2\alpha+s-1}dy=
$$
$$
=\frac{1}{2^{2\alpha}}\,\,\Gamma\left[\begin{array}{cc}
 $$\frac{s}{2},$$ & $$\frac{s}{2}-\frac{\nu-1}{2}$$ \\
$$\alpha+\frac{s}{2}-\frac{\nu-1}{2},$$ & $$\alpha+\frac{s}{2}$$  \\
                                                           \end{array}
                                                         \right] f^*(2\alpha+s).
$$

Найдем теперь преобразование Меллина от дробной производной Бесселя на полуоси
$$
((DB_\nu^\alpha f)(x))^*(s)=((B_\nu^n(IB_{\nu,-}^{n-\alpha}f(x))^*(s)=
$$
$$
=2^{2n}\Gamma\left[
               \begin{array}{cc}
                 n+1-\frac{s}{2} & \frac{1-s+\nu}{2}+n \\
                 1-\frac{s}{2} & \frac{1-s+\nu}{2} \\
               \end{array}
             \right]((IB_{\nu,-}^{n-\alpha}  f(x))^*(s-2n)=
$$
$$
=2^{2n}\Gamma\left[
               \begin{array}{cc}
                 n+1-\frac{s}{2} & \frac{1-s+\nu}{2}+n \\
                 1-\frac{s}{2} & \frac{1-s+\nu}{2} \\
               \end{array}
             \right]\times
$$
$$
\times\frac{1}{2^{2(n-\alpha)}}\,\,\Gamma\left[\begin{array}{cc}
$$\frac{s}{2}-n,$$ & $$\frac{s}{2}-n-\frac{\nu-1}{2}$$ \\
$$n-\alpha+\frac{s}{2}-n-\frac{\nu-1}{2},$$ & $$n-\alpha+\frac{s}{2}-n$$  \\
\end{array} \right] (f(x))^*(s-2\alpha)=
$$
$$
=2^{2\alpha}\Gamma\left[
               \begin{array}{cc}
                 n+1-\frac{s}{2} & \frac{1-s+\nu}{2}+n \\
                 1-\frac{s}{2} & \frac{1-s+\nu}{2} \\
               \end{array}
             \right]\,\Gamma\left[\begin{array}{cc}
$$\frac{s}{2}-n,$$ & $$\frac{s}{2}-n-\frac{\nu-1}{2}$$ \\
$$\frac{s}{2}-\alpha-\frac{\nu-1}{2},$$ & $$\frac{s}{2}-\alpha$$  \\
\end{array} \right] (f(x))^*(s-2\alpha).
$$
Применяя формулу
$$
 \Gamma (1-z)\Gamma (z)={\pi \over \sin {(\pi z)}},\qquad z\not \in \mathbb {Z},
 $$
 в числителе получим
$$
\Gamma\left(1+n-\frac{s}{2}\right)\Gamma\left(\frac{s}{2}-n\right)=\frac{\pi}{ \sin (\frac{s}{2}-n)\pi}=\frac{(-1)^n\pi}{ \sin (\frac{s}{2})\pi},
$$
$$
\Gamma\left(\frac{1-s+\nu}{2}+n\right)\Gamma\left(\frac{s-\nu+1}{2}-n\right)=\Gamma\left(1-\frac{1-s+\nu}{2}-n\right)\Gamma\left(\frac{1-s+\nu}{2}+n\right)=$$
$$=\frac{\pi}{ \sin (\frac{1-s+\nu}{2}+n)\pi}=
\frac{(-1)^n\pi}{ \sin (\frac{1-s+\nu}{2})\pi}.
$$
Таким образом,
$$
\frac{(-1)^n\pi}{\Gamma\left(\frac{1-s+\nu}{2}\right) \sin (\frac{1-s+\nu}{2})\pi}=(-1)^n\Gamma\left(\frac{1+s-\nu}{2}\right),
$$
$$
\frac{(-1)^n\pi}{\Gamma\left(1-\frac{s}{2}\right) \sin (\frac{s}{2})\pi}=(-1)^n\Gamma\left(\frac{s}{2}\right).
$$
Доказательство закончено.
\end{proof}

\begin{theorem} Для дробного интеграла и дробной производной Бесселя на полуоси при $\alpha,\beta>0$ справедливо полугрупповое свойство
\begin{equation}\label{SemiGroup3}
    IB_{\nu,-}^{-\alpha}IB_{\nu,-}^{\beta}f= IB_{\nu,-}^{\alpha+\beta}f,
\end{equation}
\begin{equation}\label{SemiGroup4}
    DB_{\nu,-}^{\alpha}DB_{\nu,-}^{\beta}f= DB_{\nu,-}^{\alpha+\beta}f.
\end{equation}
\end{theorem}
\begin{proof} Пусть $g(y)=(IB_{\nu,-}^{\beta}f)(y)$. Используя \eqref{Mellin2}, мы получим
$$
 (IB_{\nu,-}^{\alpha}[(IB_{\nu,-}^{\beta}f)(y)](x))^*(s)=((IB_{\nu,-}^{\alpha}g)(x))^*(s)=
 $$
 $$
 =\frac{1}{2^{2\alpha}}\,\,\Gamma\left[\begin{array}{cc}
  $$\frac{s}{2},$$ & $$\frac{s}{2}-\frac{\nu-1}{2}$$ \\
$$\alpha+\frac{s}{2}-\frac{\nu-1}{2},$$ & $$\alpha+\frac{s}{2}$$  \\
                                                           \end{array}
                                                         \right] g^*(2\alpha+s)=
 $$
 $$
 = \frac{1}{2^{2\alpha}}\,\,\Gamma\left[\begin{array}{cc}
 $$\frac{s}{2},$$ & $$\frac{s}{2}-\frac{\nu-1}{2}$$ \\
$$\alpha+\frac{s}{2}-\frac{\nu-1}{2},$$ & $$\alpha+\frac{s}{2}$$  \\
                                                           \end{array}
                                                         \right]
 (IB_{\nu,-}^{\beta}f)^*(2\alpha+s)=
$$
 $$
 = \frac{1}{2^{2\alpha}}\,\,\Gamma\left[\begin{array}{cc}
 $$\frac{s}{2},$$ & $$\frac{s}{2}-\frac{\nu-1}{2}$$ \\
$$\alpha+\frac{s}{2}-\frac{\nu-1}{2},$$ & $$\alpha+\frac{s}{2}$$  \\
                                                           \end{array}
                                                         \right]\times$$
                                                         $$\times\frac{1}{2^{2\beta}}\,\,\Gamma\left[\begin{array}{cc}
 $$\frac{2\alpha+s}{2},$$ & $$\frac{2\alpha+s}{2}-\frac{\nu-1}{2}$$ \\
$$\beta+\frac{2\alpha+s}{2}-\frac{\nu-1}{2},$$ & $$\beta+\frac{2\alpha+s}{2}$$  \\
                                                           \end{array}
                                                         \right]
f^*(2\alpha+2\beta+s).
$$
С другой стороны
$$
(IB_{\nu,-}^{\alpha+\beta}f)^*(s)=$$
$$=\frac{1}{2^{2(\alpha+\beta)}}\,\,\Gamma\left[\begin{array}{cc}
 $$\frac{s}{2},$$ & $$\frac{s}{2}-\frac{\nu-1}{2}$$ \\
$$\alpha+\beta+\frac{s}{2}-\frac{\nu-1}{2},$$ & $$\alpha+\beta+\frac{s}{2}$$  \\
                                                           \end{array}
                                                         \right] f^*(2\alpha+2\beta+s).
$$
Мы имеем верное равенство
$$
\frac{1}{2^{2\alpha}}\Gamma\left[\begin{array}{cc}
 $$\frac{s}{2},$$ & $$\frac{s}{2}-\frac{\nu-1}{2}$$ \\
$$\alpha+\frac{s}{2}-\frac{\nu-1}{2},$$ & $$\alpha+\frac{s}{2}$$  \\
                                                           \end{array}
                                                         \right]{\times}\frac{1}{2^{2\beta}}\,\,\Gamma\left[\begin{array}{cc}
 $$\frac{2\alpha+s}{2},$$ & $$\frac{2\alpha+s}{2}-\frac{\nu-1}{2}$$ \\
$$\beta+\frac{2\alpha+s}{2}-\frac{\nu-1}{2},$$ & $$\beta+\frac{2\alpha+s}{2}$$  \\
                                                           \end{array}
                                                         \right]{=}
$$
$$
{=}\frac{1}{2^{2(\alpha+\beta)}}\,\,\Gamma\left[\begin{array}{cc}
 $$\frac{s}{2},$$ & $$\frac{s}{2}-\frac{\nu-1}{2}$$ \\
$$\alpha+\beta+\frac{s}{2}-\frac{\nu-1}{2},$$ & $$\alpha+\beta+\frac{s}{2}$$  \\
                                                           \end{array}
                                                         \right].
$$
Таким образом, справедливо полугрупповое свойство \eqref{SemiGroup3}. Свойство \eqref{SemiGroup4} доказывается аналогично.
\end{proof}

Отметим, что в литературе по интегральным преобразованиям с параметром часто полугрупповое свойство называется индексным законом.

В заключение отметим, что конструкции дробных степеней оператора Бесселя вида \eqref{Bess3} и \eqref{DrobessDer} могут быть применены к исследованию различных дифференциальных уравнений дробного порядка, а также в теории операторов преобразования, см. \cite{Sita3}--\cite{Sita10}, \cite{SS1}--\cite{SS2}. Дифференциальное уравнение с  дробной степенью оператора Бесселя используется при моделировании случайного блуждания частицы (см. \cite{Garra1}--\cite{Garra2}). Другой подход к определению дробных степеней оператора Бесселя см. в \cite{McBride}. Исследованию различных степеней оператора гиперболического типа, содержащего операторы Бесселя, посвящены работы \cite{Sh1}--\cite{Sh4}.

\end{document}